\documentclass[12pt,british,reqno]{amsart}

\usepackage{babel}

\usepackage{amsthm,amssymb,amsmath} 
\usepackage[latin1]{inputenc}
\usepackage{tikz}
\usepackage{tikz-cd}
\usepackage{enumerate}

\usepackage{mathtools}

\usepackage[all]{xy}

\makeatletter
\def\@map#1#2[#3]{\mbox{$#1 \colon\thinspace #2 \longrightarrow #3$}}
\def\map#1#2{\@ifnextchar [{\@map{#1}{#2}}{\@map{#1}{#2}[#2]}}
\g@addto@macro\@floatboxreset\centering
\makeatother

\renewcommand{\epsilon}{\ensuremath{\varepsilon}}

\renewcommand{\to}{\ensuremath{\longrightarrow}}
\renewcommand{\mapsto}{\ensuremath{\longmapsto}}

\newtheoremstyle{theoremm}{}{}{\itshape}{}{\scshape}{.}{ }{}
\theoremstyle{theoremm}
\newtheorem{thm}{Theorem}
\newtheorem{lem}[thm]{Lemma}
\newtheorem{prop}[thm]{Proposition}

\newtheoremstyle{remark}{}{}{}{}{\scshape}{.}{ }{}

\theoremstyle{remark}

\newtheorem{rem}[thm]{Remark}

\renewcommand{\ker}[1]{\ensuremath{\operatorname{\text{Ker}}\left({#1}\right)}}

\newcommand{\reth}[1]{Theorem~\protect\ref{th:#1}}

\numberwithin{equation}{section}
\allowdisplaybreaks

\usepackage[pdftex,%
bookmarks=true,
breaklinks=true,
bookmarksnumbered = true,%     % Signets numÃ?Â?Ã?Â©rotÃ?Â?Ã?Â©s
colorlinks= true,%     % Liens en couleur
urlcolor= green,
anchorcolor = yellow,
citecolor=blue,%  % Couleur des liens externes
]{hyperref}                   % Utilisation de HyperTeX

\begin{document}

\title[The Borsuk-Ulam Theorem for $n$-valued maps between surfaces]{The Borsuk-Ulam Theorem \\ for $n$-valued maps between surfaces} 
\author[V.~C.~Laass]{Vinicius Casteluber Laass}
\address{VINICIUS CASTELUBER LAAASS \newline
Universidade Federal da Bahia, IME, Departamento de Matemática \newline
 Av. Milton Santos, S/N Ondina CEP: 40170-110, Salvador-BA, Brazil.}
\email{vinicius.laass@ufba.br}
\author[C.~M.~Pereiro]{Carolina de Miranda e Pereiro}
\address{CAROLINA DE MIRANDA E PEREIRO \newline
Universidade Federal do Esp\'{i}rito Santo, UFES, Departamento de Matem\'{a}tica \newline
29075-910, Vit\'{o}ria, Esp\'{i}rito Santo, Brazil}
\email{carolina.pereiro@ufes.br}

%\subjclass[2010]{Primary: 20F36, 20H15; Secondary: 57N16.}
%
\date{\today}

%\keywords{Surface braid groups, crystallographic group, flat manifold, Anosov diffeomorphism, K\"ahler manifold}

%\date{\today}

\begin{abstract}
%\noindent 
In this work we analysed the validity of a type of Borsuk-Ulam theorem for multimaps between surfaces. We developed an algebraic technique involving braid groups to study this problem for $n$-valued maps. As a first application we described when the Borsuk-Ulam theorem holds for splits and non-splits multimaps $\phi \colon X \multimap Y$ in the following two cases: $(i)$ $X$ is the $2$-sphere eqquiped with the antipodal involution and $Y$ is either a closed surface or the Euclidean plane; $(ii)$ $X$ is a closed surface different of the $2$-sphere eqquiped with a free involution $\tau$ and $Y$ is the Euclidean plane. The results are exhaustive and in the case $(ii)$ are described in terms of an algebraic condition involving the first integral homology group of the orbit space $X / \tau$.
\end{abstract}

\maketitle
\let\thefootnote\relax
\section{Introduction}

The celebrated Borsuk-Ulam theorem states that for any continuous map $f\colon\mathbb S^{m}\rightarrow \mathbb R^{m}$ there exists a point $x \in \mathbb S^{m}$ such that $f(x)=f(-x)$~\cite[Satz II]{Bo}. More information about the history, some generalisations and applications of these result may be found for example in the book of Matou\v{s}ek~\cite{Mato}, as well as in the paper~\cite{Sten} of Steinlein published in 1985, which was the most complete survey of this subject at that time.

In 1983, H. Schirmer proved a version of the Borsuk-Ulam theorem for a family of $n$-valued maps $\phi \colon X \multimap \mathbb R^{m}$, where $X$ is a compact metric space, simply connected and has the homology of $\mathbb S^{m}$ (see~\cite[Theorem~2 and Corollary]{Sc1}). Others authors were also interested in discuss antipodes problems for multivalued maps and there are some results in this direction as we can see in~\cite{vHS} and~\cite{Izy} for example. On the other hand, the fixed point problem for multivalued maps between surfaces were studied in a recent work of D. L. Gonçalves and J. Guaschi using a braid approach (see~\cite{GG1}).

Motivated by the previous works, the main goal of this paper is to establish a version of the Borsuk-Ulam theorem for $n$-valued maps between certain surfaces and we develop a new perspective using configuration spaces and braid groups as tools. Before to enunciate more precisely the results, we recall some definitions and properties of multivalued maps.

%\textcolor{red}{!!! V !!! uniformizar a terminologia: multivalued function, multifunction... aqui e no restante do texo}

%{\color{blue} !!C!! acrescentar (or a multifunction)}

A multivalued function (or a multifunction) $\phi\colon X\multimap Y$  is a correspondence that associates to each $x\in X$ a non-empty subset $\phi(x)$ of $Y$. Following~\cite{GG1} a multifunction $\phi$ is \emph{upper semi-continous} if for all $x\in X$, $\phi(x)$ is closed, and given an open set $V$ in $Y$, the set $\left\lbrace x\in X\,:\,\phi(x)\subset V\right\rbrace $ is open in $X$, and $\phi$ is \emph{lower semi-continuous} if the set $\left\lbrace  x\in X\,:\, \phi(x)\cap V\neq \emptyset\right\rbrace$ is open in $X$, and is \emph{continous} if it is upper semi-continuous and lower semi-continuous. An \emph{$n$-valued map} (or multimap) $\phi:X\multimap Y$ is a continuous multifuction that to each $x\in X$ assigns a unordered subset of $Y$ of cardinal exactly $n$. Schirmer~\cite{Sc2} defined that a $n$-valued map $\phi :X\multimap Y$ is \emph{split} if there exists single-valued maps $f_{1},\ldots,f_{n}:X\rightarrow Y$ such that $\phi(x)=\left\lbrace f_{1}(x),\ldots,f_{n}(x)\right\rbrace $ for all $x\in X$. In particular, all split $n$-valued maps are continuous if $Y$ is an Hausdorff topological space~\cite[Propositon 42]{GG1}.

A convenient way of dealing with $n$-valued function is making use of configuration spaces, as explored by R. F. Brown and D. L. Gonçalves~\cite{BG} and also by the latter author and J. Guaschi~\cite{GG1}, and we describe in the following. Given a topological space $Y$, we recall that % $F_{n}(Y)$  , defined by:
\[F_{n}(Y)=\left\lbrace (y_{1},\ldots,y_{n})\in Y^{n} \,:\, y_{i}\neq y_{j}\,\,\mbox{if}\,\, i\neq j\right\rbrace\] is the \emph{$n^{th}$ ordered configuration space} of $Y$. The symmetric group $S_{n}$ on $n$ elements acts freely on $F_{n}(Y)$ by permuting the coordinates. The corresponding quotient space is denoted by $D_{n}(Y)=F_{n}(Y)/S_{n}$ and it is called the \emph{$n^{th}$ unordered configuration space} of $Y$. 

There is a natural correspondence between a $n$-valued map $\phi \colon X\multimap Y$ and a single map $\Phi \colon X\rightarrow D_{n}(Y)$, if we consider $\Phi(x)=\phi(x)$. In the particular case where $X$ and $Y$ are metric spaces, it is helpful to use this correspondence because the multimap $\phi$ is continuous if, and only if, the single-valued map $\Phi$ is also continuous by~\cite[Theorem 8]{GG1}.

Now, we are able to define the Borsuk-Ulam property for multivalued maps. Let $X$ and $Y$ be topological spaces such that $X$ admits a free involution $\tau$. Given a multivalued map $\phi\colon X\multimap Y$, we define a \emph{Borsuk-Ulam coincidence of $\phi$ with respect to $\tau$} as a pair  $\left\{x,\tau(x)\right\}$ of points of $X$ such that $\phi(x)\cap\phi(\tau(x))\neq \emptyset$. The behavior of the $n$-valued maps is divided between the distinct classes splits and non-splits and since we are interested in classify these maps in relation to Borsuk-Ulam coincidences, we define that the triple $(X,\tau;Y)$ has the \emph{$n$-split Borsuk-Ulam property} (resp.\ \emph{$n$-non-split Borsuk-Ulam property}) if every split (resp.\ non-split) $n$-valued map from $X$ to $Y$ has a Borsuk-Ulam coincidence with respect to $\tau$. Finally we say that the triple $(X,\tau;Y)$ has the \emph{$n$-Borsuk-Ulam property} if it has the $n$-split and $n$-non-split Borsuk-Ulam property. In the remaining of the text, we abbreviate Borsuk-Ulam property by BUP.

Note that all $1$-valued maps are split and the above definition coincides with the BUP for single-valued maps in the same sense that were used in the works~\cite{DPV,MatMonSan}, for example. In the particular cases where $X$ is a closed surface and $Y$ is either a closed surface or the Euclidean plane, the complete classification of the triples $(X,\tau;Y)$ that possess the $1$-BUP can be found in~\cite{G,GG0}.

Some results about Borsuk-Ulam coincidences for $n$-valued maps are straightforward consequence from the single-valued case. For example, if the triple $(X,\tau;Y)$ has the $1$-BUP, then any split $n$-valued map $\phi=\left\{f_{1},\ldots,f_{n}\right\} \colon X\multimap Y$ has Borsuk-Ulam coincidence, for all $n$. This follows with the same proof given by H. Schirmer~\cite[Theorem 2]{Sc1}. Nevertheless, we can not conclude the same for non-split $n$-valued maps, as we shall see in~\reth{result2}, where we have a triple $(X,\tau;Y)$ that satisfies the $1$-BUP but does not satifies the $2$-BUP, for example. On the other hand, if $\tau$ is the antipodal map on the $2$-sphere, then the triple $(\mathbb S^{2}, \tau; \mathbb S^{2})$ is an example that does not satisfies the $1$-BUP although it satisfies de $n$-BUP for all $n\geq 2$, which is one of the consequences of the following main theorem.

\begin{thm}\label{th:result1}
	Let $n$ be a positive integer, let $\tau\colon \mathbb{S}^2 \to \mathbb{S}^2$ be the antipodal map and let $Y$ be $\mathbb{R}^2$ or a closed surface. The triple $(\mathbb{S}^2,\tau;Y)$ has the $1$-BUP if, and only if, $Y \neq \mathbb{S}^2$ and it has the $n$-BUP for all $n \geq 2$.
\end{thm}

The above result is a complete classification due the fact that the antipodal map is essentially the only free involution on $\mathbb{S}^2$ (see for example~\cite[Theorem~1]{Asoh}).

Besides the theorem above, we also give a full classification of all triples $(X,\tau,\mathbb R^{2})$ in relation to BUP, where $X$ is a closed surface. Before to state the other main result of this paper let us describe some facts and notation. Suppose that $X$ is a pathwise connected $CW$-complex equipped with a free cellular involution $\tau$. We say that two points $x,y \in X$ are equivalent if $y \in \{ x,  \tau(x) \}$. We denote the correspondent quotient space by $X_\tau$, which we also call the orbit space, and the natural projection by $p_\tau \colon X \to X_\tau$. Note that $X_\tau$ is also a pathwise $CW$-complex and $p_\tau$ is a double covering map. Therefore we have the following short exact sequence:
\begin{equation}\label{seq:theta}
\xymatrix{
1 \ar[r] & \pi_1 (X) \ar[r]^-{(p_\tau)_\#} & \pi_1(X_\tau) \ar[r]^-{\theta_\tau} & \dfrac{\pi_1(X_\tau)}{(p_\tau)_\#\left(\pi_1 (X)\right)} \cong \mathbb{Z}_2 = \{ \overline{0},\overline{1} \} \ar[r] & 1, 
}
\end{equation}
where $\theta_\tau$ is the natural projection. The homomorphism $\theta_\tau$ has a factorization to the abelianization of $\pi_1 (X_\tau)$ which is the $1$-homology group of $X_\tau$. We denote this factorization by $\hat{\theta_\tau} \colon H_1 (X_\tau) \to \mathbb{Z}_2$. If $X_\tau$ is a closed non-orientable surface, we denote by $\hat{\delta}$ the only non-trivial element of order two of $H_1(X_\tau)$. We  obtain the next result that extends the $1$-BUP case due to D. L. Gonçalves~\cite{G}.

\begin{thm}\label{th:result2}
	Let $n \geq 1$ and $X$ be a closed surface different of $\mathbb{S}^2$ which admits a free involution $\tau \colon X\rightarrow X$. The following conditions are true:
	\begin{enumerate}
		\item If $X_\tau$ is a closed orientable surface or $X_\tau$ is a closed non-orientable surface and $\hat{\theta_\tau}(\hat{\delta})=\overline{0}$, then the triple $(X,\tau;\mathbb{R}^2)$ does not have the $n$-split BUP and it does not have the $n$-non-split BUP.%there exists a split and a non-split \textbf{!!C!! não sei se tem o traço: non split, tambem acho que tem que tirar o "a"} $n$-valued maps from $X$ to $\mathbb{R}^2$ without Borsuk-Ulam coincidence.
		
		\item If $X_\tau$ is a closed non-orientable surface, $\hat{\theta_\tau}(\hat{\delta})=\overline{1}$ and $n$ is even, then the triple $(X,\tau;\mathbb R^{2})$ has the $n$-split BUP, but it does not have the $n$-non-split BUP.% all splits $n$-valued maps from $X$ to $\mathbb{R}^2$ have Borsuk-Ulam coincidence and there exists a non-split \textbf{!!C!!} $n$-valued map from $X$ to $\mathbb{R}^2$ without Borsuk-Ulam coincidence.
		
		\item If $X_\tau$ is a closed non-orientable surface, $\hat{\theta_\tau}(\hat{\delta})=\overline{1}$ and $n$ is odd, then the triple $(X,\tau;\mathbb{R}^2)$ has the $n$-BUP.
	\end{enumerate}
\end{thm}

We emphasize to the reader that the above result is complete in the sense that it is possible to know when a triple $(X,\tau;\mathbb{R}^2)$ has (or not) the $n$-BUP. Also, we analyzed the behavior of $n$-split and $n$-non-split multimaps in relation to BUP separately.

We also obtain partial results about the classfication of the triples $(X,\tau;Y)$ in relation to the $n$-BUP, where $X$ and $Y$ are both closed surfaces, $X \neq \mathbb{S}^2$. These results are rather straightforward and it is a consequence of Theorem~\ref{th:result2} and the classification of the triples in relation to $1$-BUP given in~\cite{GG0}. We enunciate and prove these examples in Section~\ref{sc:examples}. The study to obtain a full classification of the above mentioned triples in relation to $n$-split BUP and $n$-non-split BUP is a work in progress.

Besides the introduction and Section~\ref{sc:examples}, the manuscript is organised as follows. In Section~\ref{sec:2} we give a brief resume of the braid groups of surfaces. We define some important subgroups which is usefull to obtain \reth{diagram}. This result gives a necessary and suficient algebraic condition to study the BUP for $n$-valued maps  whose domain is a CW-complex eqquiped with a free cellular involution and the target is a closed surface. In Section~\ref{sec:3} we derive some properties for the subgroups mentioned above when the surface is the Euclidean plane which is usefull to prove \reth{result2}. Section \ref{sec:res1} of the paper is devoted to proving \reth{result1} and the proof will follow from Propositions \ref{propA}--\ref{propC}. We prove \reth{result2} at Section \ref{sec:res2} and the proof will follow from Propositions \ref{prop1}--\ref{prop3}.

\section{Preliminaries and generalities}\label{sec:2}

Let $n$ be a positive integer and let $Y$ be either a closed surface or $\mathbb R^{2}$. The braid groups and the pure braid groups of the plane were defined by Artin~\cite{A1,A2} and generalized for surfaces by Fox and Neuwirth~\cite{FoN} using configuration spaces. The fundamental group $\pi_{1}(D_{n}(Y))$ (resp. $\pi_{1}(F_{n}(Y))$) is the $n$-string braid group (resp. $n$-string pure braid group) of $Y$ and it is denoted by $B_{n}(Y)$ (resp. $P_{n}(Y)$). This give rise to the following short exact sequence induced by the $\left( n! \right)$-fold covering map $F_n(Y) \to D_n(Y)$:
\begin{equation}\label{seq:1}
	\xymatrix{
		1 \ar[r] & P_{n}(Y) \ar[r] & B_{n}(Y) \ar[r] & S_{n} \ar[r] & 1.
	}		
\end{equation}

The notion of \emph{mixed braid groups} was defined in~\cite{GG2,GG3} taking an intermediary quotient of the configuration space $F_{m+n}(Y)$ by the free action of the subgroup $S_{m}\times S_{n}\subset S_{m+n}$. The orbit space $F_{m+n}(Y)/{S_{m}\times S_{n}}$ is denoted by $D_{m,n}(Y)$ and  the mixed braid groups is defined by $B_{m,n}(Y)=\pi_{1}(D_{m,n}(Y))$. Therefore, we have the following short exact sequence induced by the $\left( \left( m! \right) \times \left( n! \right)\right)$-fold covering map $F_{m,n}(Y) \to D_{m,n}(Y)$:
\begin{equation}\label{seq:1.1}
\xymatrix{
	1 \ar[r] & P_{m,n}(Y) \ar[r] & B_{m,n}(Y) \ar[r] & S_{m}\times S_{n} \ar[r] & 1.
}		
\end{equation}

An arbitrary element of $D_{m,n}(Y)$ can be seen as a pair $(\bar{x},\bar{y})$, where $\bar{x}$ and $\bar{y}$ are disjoints subsets of $Y$ of cardinality $m$ and $n$, respectively. If $m=n$, let $D^{2}_{n,n}(Y)$ be the orbit space of the free involution $\hat{\tau}:D_{n,n}(Y)\rightarrow D_{n,n}(Y)$ defined by $\hat{\tau}(\bar{x},\bar{y})=(\bar{y},\bar{x})$. % the group $\mathbb Z_{2}$ acts freely in $D_{n,n}(Y)$ by permuting $\bar{x}$ with $\bar{y}$. the coordinates, 
Therefore, we obtain another important subgroup of $B_{2n}(Y)$ defined by $B^{2}_{n,n}(Y)=\pi_{1}(D^{2}_{n,n}(Y))$, which are related by the following short exact sequence induced by the double covering map $D_{n,n}(Y) \to D_{n,n}^2(Y)$:
\begin{equation}\label{seq:2}
\xymatrix{
		1 \ar[r] & B_{n,n}(Y) \ar[r] & B^{2}_{n,n}(Y) \ar[r]^-{\pi} & \mathbb Z_{2} \ar[r] & 1.
}		
\end{equation}

Theses subgroups of $B_{2n}(Y)$ are important for this work, and they will appear in the following theorem. The result below is a generalization of~\cite[Proposition~13]{GG0} and it will play a key rôle in the rest of the paper.

\begin{thm}\label{th:diagram} Let $n$ be a positive integer, let $X$ be a pathwise connected and locally compact $CW$-complex eqquiped with a cellular free involution $\tau$ and let $Y$ be either $\mathbb R^{2}$ or a closed surface. Suppose that the $n$-split BUP (resp.\ n-non-split BUP) does not hold for the triple $(X,\tau;Y)$. Then there exists a homomorphism $\Psi \colon \pi_{1}(X_\tau)\rightarrow B^{2}_{n,n}(Y)$ such that $\Psi ( \ker{ \theta_\tau}) \subset P_{2n}(Y)$ (resp.\ $\Psi ( \ker{ \theta_\tau}) \not\subset P_{2n}(Y)$) and the following diagram is commutative:
				\begin{equation}\label{diagram}
					\begin{gathered}\xymatrix{
							\pi_1(X_\tau) \ar@{.>}[rr]^{\Psi} \ar[rd]_{\theta_\tau} && B^{2}_{n,n}(Y) \ar[ld]^{\pi} && \\
							& \mathbb{Z}_2 .& & &
				}\end{gathered}\end{equation}
Conversely, if such a factorization $\Psi$ of the diagram~(\ref{diagram}) exists and it satisfies $\Psi ( \ker{ \theta_\tau}) \subset P_{2n}(Y)$ (resp.\ $\Psi ( \ker{ \theta_\tau}) \not\subset P_{2n}(Y)$), then the $n$-split BUP (resp.\ $n$-non-split BUP) does not hold for the triple $(X,\tau;Y)$ in the following cases:
\begin{enumerate}[(a)]

\item The dimension of the CW-complex $X$ is less than or equal to two and $Y=\mathbb{S}^2$.

\item The dimension of the CW-complex $X$ is less than or equal to three and $Y = \mathbb{RP}^2$.

\item $Y$ is different from $\mathbb{S}^2$ and $\mathbb{RP}^2$.

\end{enumerate}
\end{thm}
\begin{rem}
The Theorem~\ref{th:diagram} gives us a necessary and sufficient algebraic condition to investigate the validity of the BUP for $n$-valued maps from $X$ to $Y$ in the cases where $X$ is a closed surface and $Y$ is a closed surface or the Euclidean plane.	
\end{rem}	

\begin{proof}[Proof of Theorem~\ref{th:diagram}] First of all, the hypothesis over $X$ implies that $X$ is a metric space by~\cite[Chapter~IV, Proposition~5.8 and Theorem~5.14]{Hu}. Since $Y$ is the Euclidean plane or a closed surface, then $Y$ is also a metric space. Thus, we can consider $n$-valued maps from $X$ to $Y$ as single maps from $X$ to $D_n(Y)$ by~\cite[Theorem 8]{GG1}.

Suppose first that the triple $(X,\tau;Y)$ does not have the $n$-BUP which means that there exists a map $\Phi\colon X \to D_n(Y)$ without Borsuk-Ulam coincidence. Then, the map $\hat{\Phi} \colon X \to D_{n,n}(Y)$ given by $\hat{\Phi}(x) = (\Phi(x),\Phi(\tau(x)))$ is well defined. Note that $\hat{\Phi}$ is $(\tau,\hat{\tau})$-equivariant, where $\hat{\tau}$ is the free involution on $D_{n,n}(Y)$ defined by $\hat{\tau}(\overline{x},\overline{y}) = (\overline{y},\overline{x})$. Using standard covering spaces arguments analougous that was used in the first part of the proof of~\cite[Lemma~5]{GGL} there exists an homomorphism $\Psi:\pi_{1}(X_\tau)\rightarrow B^{2}_{n,n}(Y)$  such that the following diagram is commutative:
\begin{equation}\label{diagram_proof}
\begin{gathered}\xymatrix{
\pi_1(X) \ar[rr]^-{\hat{\Psi}} \ar[d]_-{(p_\tau)_\#} && B_{n,n}(Y) \ar[d] \\
\pi_1(X_\tau) \ar[rr]^-{\Psi} \ar[rd]_{\theta_\tau} && B^{2}_{n,n}(Y) \ar[ld]^{\pi} && \\
							& \mathbb{Z}_2 ,& & &
				}\end{gathered}\end{equation}
where $\hat{\Psi} = \hat{\Phi}_\#$. If the triple $(X,\tau;Y)$ does not have the $n$-split BUP, then we may suppose without loss of generality that the multimap $\Phi$ is split. This occurs if, and only if, there exists maps $f_1, \ldots f_n \colon X \to Y$ such that $\hat{\Phi}(x) = \left( \{ f_1(x) , \ldots, f_n(x) \} , \{ f_1(\tau(x)) , \ldots, f_n(\tau(x)) \} \right)$ for all $x \in X$. The above equality is equivalent that the map $\hat{\Phi}$ has a lift with respect to the covering $F_{2n}(Y) \to D_{n,n}(Y)$ which is algebraically equivalent to $P_{2n}(Y) \supset {\rm Im}\left( \hat{\Psi} \right) = \Psi( \left( {\rm Im } (p_\tau)_\# \right) = \Psi \left( \ker{\theta_\tau} \right)$. Analogous preceding arguments show that $\Psi ( \ker{ \theta_\tau}) \not\subset P_{2n}(Y)$ if the triple $(X,\tau;Y)$ does not have the $n$-non-split BUP. The diagram~(\ref{diagram}) is the bottom of diagram~(\ref{diagram_proof}) so we proved the first part of the theorem.

Now, we will prove the converse part of the theorem. We assert that there exists a map $\overline{\Phi}\colon\thinspace X_\tau \to D^2_{n,n}(Y)$ that induces $\Psi$. We treat the three cases of the statement in turn. Before to do this, note that the dimension of $X_\tau$ as a CW-complex is the same as $X$ because $\tau$ is free and cellular.

\begin{enumerate}[(a)]
	
\item\label{it:a} The existence of the map $\overline{\Phi}$ follows using classical arguments of obstruction theory in the same manner that were used in~\cite[Proof of Proposition 13, item (a)]{GG0}.
	
\item\label{it:b} The universal covering of $D^2_{n,n}(Y)$ is the same of $F_{2n}(\mathbb{RP}^2)$ which is $\mathbb{S}^3$ by~\cite[Proposition~6]{GG4}. Since $X_\tau$ has dimension less or equal to three, once more using classical obstruction theory, we may construct the map $\overline{\Phi}$.
	
\item\label{it:c} The space $F_{2n}(Y)$ is a ${\rm K}(\pi,1)$ if $Y \neq \mathbb{S}^2,\mathbb{RP}^2$ by~\cite[Chapter~4, Theorem~1.1]{FadHus} and~\cite[Corollary~2.2]{FN}, and then $D^2_{n,n}(Y)$ is also a ${\rm K}(\pi,1)$. The existence of the map $\overline{\Phi}$ is assured by~\cite[Theorem~4]{GGL}.
\end{enumerate}
Let $\hat{\Phi}\colon\thinspace X \to D_{n,n}(Y)$ be the lift of the map $\overline{\Phi} \circ p_\tau$ with respect to the covering $D_{n,n}(Y) \to D^2_{n,n}(Y)$. So we have the following commutative diagram:
\begin{equation}\label{diagram_proof_2}
\begin{gathered}\xymatrix{
X \ar[r]^-{\hat{\Phi}} \ar[d]_-{p_\tau} & D_{n,n}(Y) \ar[d] \\
X_\tau 	\ar[r]_-{\overline{\Phi}} & D^2_{n,n}(Y).
}\end{gathered}	
\end{equation}
Let $\hat{\Psi}$ be the induced homomorphism by $\hat{\Phi}$ on the level of fundamental group. So, from the diagrams (\ref{diagram}) and (\ref{diagram_proof_2}) we obtain the algebraic commutative diagram (\ref{diagram_proof}). Using similar arguments that were used in~\cite[Proofs of Lemma~5 and Theorem 7]{GGL}, we conclude that $\hat{\Phi}$ is a $(\tau,\hat{\tau})$-equivariant map and there exists a map $\Phi\colon\thinspace X \to D_n(Y)$ that satisfies $\hat{\Phi}(x) = (\Phi(x),\Phi (\tau(x)))$ for all $x \in X$. Therefore the map $\Phi$ does not have Borsuk-Ulam coincidence. Finally, using similar arguments that were used in the first part of the proof, we conclude that $\Phi$ is split (resp.\ non-split) if $\Psi ( \ker{ \theta_\tau}) \subset P_{2n}(Y)$ (resp.\ $\Psi ( \ker{ \theta_\tau}) \not\subset P_{2n}(Y)$) and the conclusion of the second part of the theorem follows. 
\end{proof}

The next result will be useful to prove \reth{result2} in Section~\ref{sec:res2}.

\begin{thm}\label{th:aumentak} Given integers numbers $n,k\geq 1$, there exists an homomorphism $\tilde{\Psi} \colon B^{2}_{n,n}(\mathbb{R}^2)\rightarrow B^2_{nk,nk}(\mathbb{R}^2)$ for which the following diagram is commutative: 
\begin{equation}\label{diag_inclusion}
\begin{gathered}\xymatrix{
B^{2}_{n,n}(\mathbb{R}^2) \ar[rr]^-{\tilde{\Psi}} \ar[rd]_-{\pi_1} && B^2_{nk,nk}(\mathbb{R}^2) \ar[ld]^-{\pi_2}\\
							& \mathbb{Z}_2 ,& 
				}\end{gathered}\end{equation}
where $\pi_1$ and $\pi_2$ are as defined in~(\ref{seq:2}).
\end{thm}

\begin{proof}
Given a point $z = (x_1,x_2,\ldots,x_n,y_1,y_2,\ldots,y_n) \in F_{2n}(\mathbb{R}^2)$, consider the real number
$$\delta(z) = \dfrac{1}{2k} \cdot \displaystyle \min_{1 \leq i < j \leq n} \{ ||x_i - x_j ||, || y_i - y_j || , || x_i - y_i || , || x_i - y_j || , || x_j - y_i || \}.$$
Note that $\delta(z) \neq 0$. Now, consider $\varepsilon(z) = (\delta(z),0) \in \mathbb{R}^2$. We define
\begin{align*}
G(z) = ( &  x_1 + 0 \cdot \varepsilon(z) , x_1 + 1 \cdot \varepsilon(z) , \ldots , x_1 + (k-1) \cdot \varepsilon(z) , \\
& x_2 + 0 \cdot \varepsilon(z) , x_2 +1 \cdot \varepsilon(z) , \ldots , x_2 + (k-1) \cdot \varepsilon(z) , \\
& \qquad  \qquad \qquad \vdots \\
& x_n + 0 \cdot \varepsilon(z) , x_n + 1 \cdot \varepsilon(z) , \ldots , x_n + (k-1) \cdot \varepsilon(z) , \\
&  y_1 + 0 \cdot \varepsilon(z) , y_1 + 1 \cdot \varepsilon(z) , \ldots , y_1 + (k-1) \cdot \varepsilon(z), \\
& y_2 + 0 \cdot \varepsilon(z) , y_2 + 1 \cdot \varepsilon(z) , \ldots , y_2 + (k-1) \cdot \varepsilon(z) , \\
& \qquad  \qquad \qquad \vdots \\
& y_n + 0 \cdot \varepsilon(z) , y_n + 1 \cdot \varepsilon(z) , \ldots , y_n + (k-1) \cdot \varepsilon(z)).
\end{align*}
Note that $G\colon F_{2n}(\mathbb{R}^2) \to F_{2nk}(\mathbb{R}^2)$ is a well defined continuous map. Using standard arguments of quotient topology, there exists maps $\overline{G}\colon D_{n,n}(\mathbb{R}^2) \to D_{nk,nk}(\mathbb{R}^2)$ and $\tilde{G}\colon D^2_{n,n}(\mathbb{R}^2) \to D^2_{nk,nk}(\mathbb{R}^2)$ such that the following diagram is commutative:
\begin{equation}\label{diag_inclusion_proof_1}\begin{gathered}\xymatrix{
F_{2n}(\mathbb{R}^2) \ar[d] \ar[r]^-{G} & F_{2nk}(\mathbb{R}^2) \ar[d] \\
D_{n,n}(\mathbb{R}^2) \ar[d] \ar[r]^-{\overline{G}} & D_{nk,nk}(\mathbb{R}^2)\ar[d] \\
D^2_{n,n}(\mathbb{R}^2) \ar[r]^-{\tilde{G}} & D^2_{nk,nk}(\mathbb{R}^2).
}\end{gathered}\end{equation}
Using covering spaces properties, similar that were used in \cite[Proof of Lemma~5]{GGL}, we may take $\tilde{\Psi} = \tilde{G}_\#\colon B^{2}_{n,n}(\mathbb{R}^2)=  \pi_1(D^2_{n,n}(\mathbb{R}^2))  \to  B^2_{nk,nk}(\mathbb{R}^2) = \pi_1 (D^2_{nk,nk}(\mathbb{R}^2)) $ to obtain the diagram~(\ref{diag_inclusion}).
\end{proof}

\section{Presentations}\label{sec:3}
In this section we present some general properties and results about the braid groups of surfaces, notably the braid group of the plane, that will be useful in this paper. As defined in the previous section, $B_{n}(\mathbb R^{2})$ is the fundamental group of the $n^{th}$ unordered configuration space $D_{n}(\mathbb R^{2})$, and it is usually denoted by $B_{n}$ and called the \emph{$n^{th}$ Artin braid group}. From now on, we will denote all the subgroups of $B_{n}$ without mention the surface $Y$ in the case that $Y=\mathbb R^{2}$.

For $n\geq2$, the group $B_{n}$ is generated by the elements $\sigma_{1},\ldots,\sigma_{n-1}$, subject to the following relations, called \emph{Artin relations}~\cite{A1}:
\[\left\{\begin{aligned}&\sigma_{i}\sigma_{i+1}\sigma_{i}=\sigma_{i+1}\sigma_{i}\sigma_{i+1} & \text{for all $1\leq i\leq n-2$;}\\ &\sigma_{j}\sigma_{i}=\sigma_{i}\sigma_{j}& \text{if $|i-j|\geq 2$ and $1\leq i,j\leq n-1$.}\end{aligned}\right.\] 
The elements of $B_{n}$, called braids, may be considered as a geometric representation of strings and braids of hair, and the generators of the Artin braid group (also called as \emph{Artin generators}) can be visualized in Figure~(\ref{fig:artin}), where each path is called string. Also, since each braid is defined in the interior of a disc, the inclusion $j:\mathbb D^{2}\rightarrow int(Y)$ allows us to see each braid in $B_{n}$ as a braid in $B_{n}(Y)$, where $Y$ is any surface.

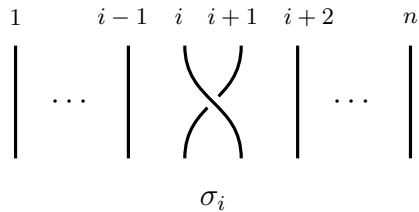
\begin{figure}[h!]%[!h]
	\hfill
	\begin{tikzpicture}[scale=0.75, very thick]
		
		\foreach \k in {5}
		{\draw (\k,3) .. controls (\k,2) and (\k-1,2) .. (\k-1,1);};
		
		\foreach \k in {4}
		{\draw[white,line width=6pt] (\k,3) .. controls (\k,2) and (\k+1,2) .. (\k+1,1);
			\draw (\k,3) .. controls (\k,2) and (\k+1,2) .. (\k+1,1);};

		\foreach \k in {1,3,6,8}
		{\draw (\k,1)--(\k,3);};
		
		\foreach \k in {2,7}
		{\node at (\k,2) {$\cdots$};};
		
		\foreach \k in {1}
		{\node at (\k,3.5) {{\scriptsize$1$}};
			\node at (\k+1.9,3.5) {{\scriptsize$i-1$}};
			\node at (\k+2.9,3.52) {{\scriptsize$i$}};
			\node at (\k+3.85,3.5) {{\scriptsize$i+1$}};
			\node at (\k+5.2,3.5) {{\scriptsize$i+2$}};
			\node at (\k+7,3.5) {{\scriptsize$n$}};};
		
		\node at (4.5,0.25) {$\sigma_{i}$};
		%\node at (15.5,0.25) {$\sigma_{i}^{-1}$};
		
	\end{tikzpicture}
	\hspace*{\fill}
	\caption{The braid $\sigma_{i}$}\label{fig:artin}
\end{figure}

There is a natural correspondence that associates to each braid a permutation of $S_{n}$, considering the projection $\pi:B_{n}\rightarrow S_{n}$, $\sigma_{i}\mapsto(i,i+1)$. The \emph{$n^{th}$ Artin pure braid group} $P_{n}=\pi_{1}(F_{n}(\mathbb R^{2}))$ is isomorphic to the kernel of $\pi$. Thus, the pure braids are the elements that has the trivial permutation. The subgroup $P_{n}$ is generated by the set $\left\{A_{i,j}\,:\,1\leq i < j \leq n\right\}$~\cite[Chapter~1, Lemma~4.2]{Ha}, where \begin{equation}\label{Aij}
A_{i,j}=\sigma_{j-1}\cdots\sigma_{i+1}\sigma^{2}_{i}\sigma^{-1}_{i+1}\cdots\sigma^{-1}_{j-1}.
\end{equation}

The \emph{Garside element} (or `half-twist') $\Delta_{n}$ of $B_{n}$ is~ defined by% {\color[rgb]{1,0,0}Figura?}
\begin{equation}\label{def:garside}
	\Delta_{n}=(\sigma_{1}\sigma_{2}\cdots\sigma_{n-1})(\sigma_{1}\sigma_{2}\cdots\sigma_{n-2})\cdots(\sigma_{1}\sigma_{2})(\sigma_{1}),
\end{equation}
and can be visualized in Figure~\ref{fig:garside} for $n=6$. \begin{figure}[h!]%[!h]
	\hfill
	\begin{tikzpicture}[scale=0.75, very thick]
		
		%% corda 6
		\draw (5,10.5)--(5,10) .. controls (5,5.5) and (0,5.5).. (0,4);
		
		%% corda 5
		\draw[white,line width=6pt] (4,10.5)--(4,10) .. controls (4,6.5) and (0,6.5).. (0,5.5);
		\draw (4,10.5)--(4,10) .. controls (4,6.5) and (0,6.5).. (0,5.5);
		\draw[white,line width=6pt] (0,5.5).. controls (0,5) and (1,5).. (1,4);
		\draw (0,5.5).. controls (0,5) and (1,5).. (1,4);

		%% corda 4
		\draw[white,line width=6pt] (3,10.5)--(3,10) .. controls (3,7.5) and (0,7.5).. (0,6.5);
		\draw (3,10.5)--(3,10) .. controls (3,7.5) and (0,7.5).. (0,6.8);
		\draw[white,line width=6pt] (0,6.8).. controls (0,6) and (2,6).. (2,5)--(2,4);
		\draw (0,6.8).. controls (0,6) and (2,6).. (2,5)--(2,4);

		%%% corda 3
		\draw[white,line width=6pt] (2,10.5)--(2,10) .. controls (2,8.5) and (0,8.5).. (0,7.8);
		\draw (2,10.5)--(2,10) .. controls (2,8.5) and (0,8.5).. (0,7.8);
		\draw[white,line width=6pt] (0,7.8).. controls (0,7) and (3,7).. (3,6)--(3,4);
		\draw (0,7.8).. controls (0,7) and (3,7).. (3,6)--(3,4);
		
		%%% corda 2
		\draw[white,line width=6pt] (1,10.5) .. controls (1,9.75) and (0,9.75).. (0,9);
		\draw (1,10.5) .. controls (1,9.75) and (0,9.75).. (0,9);
		\draw[white,line width=6pt] (0,9).. controls (0,8) and (4,8).. (4,7)--(4,4);
		\draw (0,9).. controls (0,8) and (4,8).. (4,7)--(4,4);
		%%% corda 1
		%\draw[white,line width=6pt] (0,10.5)--(0,10).. controls (0,9) and (5,9).. (5,8)--(5,4);
		%\draw (0,10.5)--(0,10).. controls (0,9) and (5,9).. (5,8)--(5,4);
		\draw[white,line width=6pt] (0,10.5).. controls (0,9) and (5,9).. (5,7.5)--(5,4);
		\draw (0,10.5).. controls (0,9) and (5,9).. (5,7.5)--(5,4);
	\end{tikzpicture}
	\hspace*{\fill}
	\caption{The Garside element $\Delta_{6}$ of $B_{6}$}\label{fig:garside}
\end{figure}
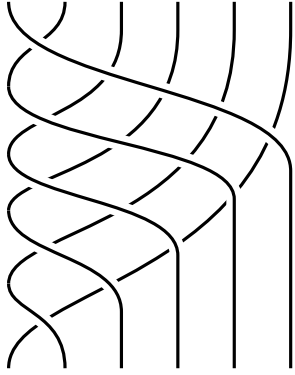
This element has an important rôle in the theory of braid group, notably in the resolution of the conjugacy problem in $B_{n}$~\cite{B2,Gar}. Using the Artin relations, it is easy to proof that 
\begin{equation}\label{Delta}
	\Delta_{n}\sigma_{i}\Delta^{-1}_{n}=\sigma_{n-i} \,\,\mbox{for\,all}\,\, 1\leq i \leq n-1.
\end{equation} 
The \emph{`full twist'} is the braid defined as the square of the half-twist and it can be written in terms of generators of $P_{n}$ as
\begin{equation}\label{full}
	\Delta^{2}_{n}=A_{1,2}(A_{1,3}A_{2,3})\cdots (A_{1,n}A_{2,n}\cdots A_{n-1,n}),  	
\end{equation} and this pure braid generates the centre of $B_{n}(Y)$ for $Y=\mathbb R^{2},\,\mathbb S^{2},\,\mathbb RP^{2}$~\cite{chow, GvB, Mu}. Those elements will be important in this paper, as we will see in the next section.

The others intermediary subgroups $B_{n,n}$ and $B^{2}_{n,n}$ of the Artin braid group also have a useful geometric representation. We can visualize the braids in $B_{n,n}$ as braids with $2n$ strings, where the first $n$ strings has as final point the first $n$ points and the last $n$  strings has as final point the last $n$ points, as in Figure~\ref{exe}(a). The group $B^{2}_{n,n}$ contains $B_{n,n}$ and also all the braids that the first $n$ strings has as final point the last $n$ points and the last $n$  strings has as final point the first $n$ points, as in Figure~\ref{exe}(b). %\textbf{o certo é $n$ first points ou first $n$ points??}
\begin{figure}[h!]%[!h][
	\hfill
	\begin{tikzpicture}[scale=0.7, very thick]
		
		\draw (1,5).. controls (1,3) and (5,3) .. (4,1);
		\draw (1,1)--(1,-1);
		\draw (4,4).. controls (4,3.5) and (2.5,3.5) .. (2.5,1.5);
		\draw (3,1).. controls (3,0.5) and (2,0.5).. (2,0);
		\draw (2,0).. controls (2,-0.5) and (3,-0.5).. (3,-1);

		\draw[white,line width=5pt] (4,1).. controls (3,-0.5) and (2,-0.5) .. (2,-1);
		\draw (4,1).. controls (3,-0.5) and (2,-0.5) .. (2,-1);
		\draw[white,line width=5pt] (2.5,1.5).. controls (2.5,0) and (4,0.5) .. (4,-1);
		\draw (2.5,1.5).. controls (2.5,0) and (4,0.5) .. (4,-1);
		
		\draw[white,line width=5pt] (4,3).. controls (4,2) and (1,3).. (1,1);
		\draw (4,3).. controls (4,2) and (1,3).. (1,1);
		
		\draw[white,line width=5pt] (1,2.5).. controls (1,1) and (3,2).. (3,1);
		\draw (1,2.5).. controls (1,1) and (3,2).. (3,1);
		
		\draw[white,line width=5pt] (1,5).. controls (1,3) and (5,3) .. (4,1);
		\draw (1,5).. controls (1,3) and (5,3) .. (4,1);
		\draw[white,line width=5pt] (3,5).. controls (3,4.5) and (4,4.5).. (4,4);
		\draw (3,5).. controls (3,4.5) and (4,4.5).. (4,4);
		\draw[white,line width=5pt] (4,5).. controls (3,4) and (1,4).. (1,2.5);
		\draw (4,5).. controls (3,4) and (1,4).. (1,2.5);
		\draw[white,line width=5pt] (2,5).. controls (2,4) and (4,4).. (4,3);
		\draw (2,5).. controls (2,4) and (4,4).. (4,3);
		
		\node at (2.5,-1.8) {(a) Braid in $B_{2,2}$};
		
		\foreach \k in {12,13,14,15}
		{\draw[white,line width=6pt] (\k,5) .. controls (\k,2) and (\k-4,2) .. (\k-4,-1);
			\draw (\k,5) .. controls (\k,2) and (\k-4,2) .. (\k-4,-1);};
		
		\foreach \k in {8,9,10,11}
		{\draw[white,line width=6pt] (\k,5) .. controls (\k,2) and (\k+4,2) .. (\k+4,-1);
			\draw (\k,5) .. controls (\k,2) and (\k+4,2) .. (\k+4,-1);};
		
		\node at (11.5,-1.8) {(b) Braid in $B^{2}_{4,4}$};
		
	\end{tikzpicture}
	\hspace*{\fill}
	\caption{Examples in $B_{n,n}$ and $B^{2}_{n,n}$}\label{exe}
\end{figure}
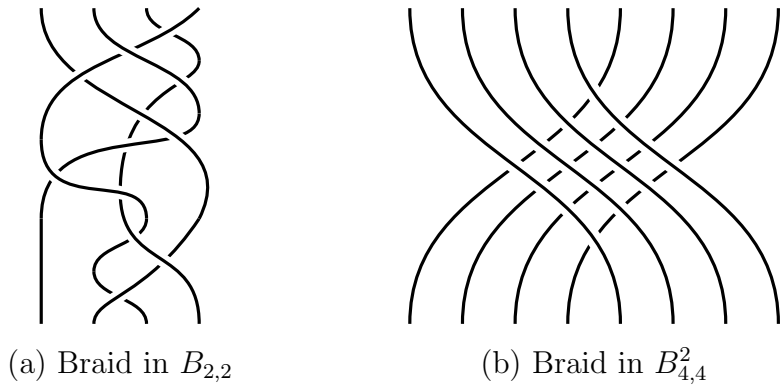

\begin{rem}\label{rem:delta}
	It is easy to see that the permutation of $\Delta_{2n}$ is $\prod^{n}_{i=1}(i,\, 2n+1-i)$, then, the element $\Delta_{2n}\in B^{2}_{n,n}\setminus B_{n,n}$. By~(\ref{seq:2}), we know that $B_{n,n}$ is a subgroup of $B^{2}_{n,n}$ of index $2$, therefore all elements in $B^{2}_{n,n} \setminus B_{n,n}$ can be written as a product of $\Delta_{2n}$ with a braid in $B_{n,n}$. 
\end{rem}

If $n=1$, we have $B_{1,1}=P_{2}$. For $n\geq2$, the following proposition give a presentation of $B_{n,n}$ that will be used in this paper, and it was based in the presentation of $B_{n,n}(\mathbb S^{2})$ given by~\cite[Proposition 10]{GG3}. 

\begin{prop}\label{prop:B_{n,n}} Let $n\geq2$, the group $B_{n,n}$ admits the following presentation:
	
	Generators: $\left\lbrace \sigma_{1},\ldots,\sigma_{n-1},\sigma_{n+1},\ldots,\sigma_{2n-1},A_{i,j}\,:\,1\leq i<j\leq 2n\right\rbrace $.
	
	\vskip.2truecm
	Relations:
	\begin{enumerate}
		\item\label{it:B1} Pure Braid relations: $A^{-1}_{r,s}A_{i,j}A_{r,s}=\left\{\begin{array}{ll}A_{i,j} \hfill i<r<s<j\mbox{\,\,\,\,or\,\,}&r<s<i<j;\\
			A_{r,j}A_{i,j}A^{-1}_{r,j}&r<i=s<j;\\
			A_{r,j}A_{s,j}A_{i,j}A^{-1}_{s,j}A^{-1}_{r,j}&i=r<s<j;\\
			A_{r,j}A_{s,j}A^{-1}_{r,j}A^{-1}_{s,j}A_{i,j}A_{s,j}A_{r,j}A^{-1}_{s,j}A^{-1}_{r,j}&r<i<s<j.\end{array}\right.$
		
		\vskip.5truecm
		\item\label{it:B2} Artin relations: $\left\{\begin{array}{lr} \sigma_{i}\sigma_{i+1}\sigma_{i}=\sigma_{i+1}\sigma_{i}\sigma_{i+1},&\quad\,1\leq i \leq n-2\quad\mbox{or}\quad n+1\leq i \leq 2n-2;\\
			\sigma_{i}\sigma_{j}=\sigma_{j}\sigma_{i},&\,|i-j|\geq2.\end{array}\right.$
		\vskip.5truecm
		\item\label{it:B3} $\sigma^{2}_{i}=A_{i,i+1},\quad\,1\leq i \leq n-1\quad\mbox{or}\quad n+1\leq i \leq 2n-1$.
		\vskip.5truecm

		\item\label{it:B4} Conjugates: $\sigma_{k}A_{i,j}\sigma^{-1}_{k}=\left\{\begin{array}{lr}A_{i,j}, & k\neq i-1, i, j-1,j;\\ 
			A_{i,j}, &k=i, j=k+1;\\
			A_{i,k+1},&k=j;\\
			A_{i,k+1}^{-1}A_{i,k}A_{i,k+1},&j=k+1, i< k;\\	
			A_{i+1,j},&i=k<j-1;\\
			A^{-1}_{k+1,j}A_{k,j}A_{k+1,j},&i=k+1.\\
		\end{array}\right.$

	\end{enumerate}
	
\end{prop}

\begin{proof}
	We use the extension method described in~\cite[Chapter~10, Proposition~1]{Jo} applyed in the short exact sequence~(\ref{seq:1.1}) for $m=n$ and $Y=\mathbb R^{2}$. According this method, a set of generators for $B_{n,n}$ is given by the union of all generators of $P_{2n}$  with a coset representatives of generators of $S_{n}\times S_{n}$. By~\cite[Chapter~1, Lemma~4.2]{Ha}, we may choose {$\{ A_{i,j} \, : \, 1 \leq i < j \leq 2n \} $} as the set of generators of $P_{2n}$. Since the transpositions generates the symetric group, we may choose $\sigma_{i}$, $i=1,\ldots,n-1,n+1,\ldots,2n-1$ as the remain generators of $B_{n,n}$. There are three types of relations. The first type is obtained by taking all the relations of the generators of $P_{2n}$ and we use again~\cite[Chapter~1, Lemma~4.2]{Ha} to obtain the relations~(\ref{it:B1}). The second type is given by rewriting the relations of $S_{n}\times S_{n}$ in terms of the chosen cosets representatives and expressing the corresponding element as words in terms of the generators of $P_{2n}$. This step can be view in a simple geometric way, which gives relations~(\ref{it:B2}) and~(\ref{it:B3}). Finally the third type is obtained by writing the conjugates of the generators of the kernel by the coset representatives as words written entirely in terms of the generators of the kernel. This rewriting process is described in details in ~\cite[Lemma~3.1]{LiWu}, which gives relation~(\ref{it:B4}). %The details are left to the reader% and the proof is equal to~\cite{GG3}.
	%Using that $A_{i,j}=\sigma_{j-1}\cdots\sigma_{i+1}\sigma^{2}_{i}\sigma^{-1}_{i+1}\cdots\sigma_{j-1}$ and the Artin relations, we obtain the relations~\ref{it:B3},~\ref{it:B4}.
\end{proof}

\begin{rem}\label{rem:Aij} Using~(\ref{Aij}), the elements $A_{i,j}\in B_{n,n}$, $1\leq i <j\leq n$ or $n+1\leq i<j\leq2n$, can be written as product of $\sigma_{k}$, $k=1,\ldots, n-1,n+1,\ldots,2n-1$. If $1\leq i \leq n$ and $n+1\leq j\leq2n$, the braid $A_{i,j}\in B_{n,n}$ can be written as \begin{equation}\label{AA2}
		A_{i,j}=(\sigma_{j-1}\cdots\sigma_{n+1})(\sigma^{-1}_{i}\cdots\sigma^{-1}_{n-1})A_{n,n+1}(\sigma_{i}\cdots\sigma_{n-1})(\sigma^{-1}_{n+1}\cdots\sigma^{-1}_{j-1}).
\end{equation}  \end{rem}

\section{Proof of~\reth{result1}}\label{sec:res1}

In this section we prove~\reth{result1} using a sequence of propositions. Recall that $\tau$ denote the antipodal map on $\mathbb{S}^2$. First, note that the triple $(\mathbb{S}^2,\tau;\mathbb{S}^2)$ does not have $1$-BUP, because the identity map of $\mathbb{S}^2$ does not collapses antipodal points. This fact with the classical Borsuk-Ulam theorem and \cite[Theorem~2 and Proposition~4]{GG0} is summarized in the following propositon:

\begin{prop}\label{propA}
Let $Y$ be $\mathbb{R}^2$ or a closed surface. Then the triple $(\mathbb{S}^2,\tau;Y)$ has the $1$-BUP if, and only if, $Y \neq \mathbb{S}^2$.
\end{prop}

With the exception of the case $Y=\mathbb S^{2}$, the next result extends the previous one and the prove uses the same idea of~\cite[Proof of Theorem~2]{Sc1}.

\begin{prop}\label{propB}
Let $Y$ be $\mathbb{R}^2$ or a closed surface different of $\mathbb{S}^2$. Then the triple $(\mathbb{S}^2,\tau;Y)$ has the $n$-BUP if $n \geq 2$.	
\end{prop}

\begin{proof}
Let $\phi\colon\thinspace \mathbb{S}^2 \multimap Y$ be $n$-valued map. Then $\phi$ is split by \cite[Lemma~13]{GG1} which means that there exists maps $f_1,f_2,\ldots,f_n\colon\thinspace \mathbb{S}^2 \to Y$ such that  $\phi(x) = \{ f_1(x),f_2(x),\ldots,f_n(x)  \}$ for all $x \in \mathbb{S}^2$. By Proposition~\ref{propA}, there exists a point $x \in \mathbb{S}^2$ such that $f_1(x)=f_1(-x)$ and therefore $\phi(x)\cap \phi(-x) \neq \emptyset$ which ends the proof.
\end{proof}

Note that \reth{result1} is a summary of Proposition \ref{propA} and \ref{propB} and the next result.

\begin{prop}\label{propC}
The triple $(\mathbb{S}^2,\tau;\mathbb{S}^2)$ has the $n$-BUP for all $n \geq 2$.	
\end{prop}

\begin{proof}
We argue by contradiction. Suppose that $(\mathbb{S}^2,\tau;\mathbb{S}^2)$ does not have the $n$-BUP for $n \geq 2$. By \reth{diagram} and the fact the quotient of $\mathbb{S}^2$ by the antipodal involution is $\mathbb{RP}^2$, we have the following commutative diagram:
$$\xymatrix{
	\pi_1(\mathbb{RP}^2) \cong \mathbb{Z}_2 \ar@{.>}[rr]^{\Psi} \ar[rd]_{\theta_\tau \cong {\rm Id}} && B^{2}_{n,n}(\mathbb{S}^2) \ar[ld]^{\pi} && \\
	& \mathbb{Z}_2 .& & &
}$$
So, $\Psi$ maps the generator of $\pi_1(\mathbb{RP}^2)$ in a braid $b \in B_{2n}(\mathbb{S}^2)$ which has order two and it belongs to $B^2_{n,n}(\mathbb{S}^2) \setminus B_{n,n}(\mathbb{S}^2)$. However, this is an absurd because the full twist $\Delta_{2n}^2$  is the only element of order two of $B_{2n}(\mathbb{S}^2)$ by~\cite[Theorem 3.5]{GvB} and $\Delta_{2n}^2 \in P_{2n} (\mathbb{S}^2) \subset B_{n,n}(\mathbb{S}^2)$.
\end{proof}

\begin{rem}
We also could prove Proposition~\ref{propB} in a similar way as Proposition~\ref{propC}, because $B_m(Y)$ is torsion-free if $Y$ is $\mathbb{R}^2$ or a closed surface different of $\mathbb{S}^2$ and for $\mathbb{RP}^2$ by \cite[Theorem~8]{FN}, and if $Y=\mathbb RP^{2}$, then the full twist $\Delta^{2}_{2n}$ is also the only element of order two of $B_{2n}(\mathbb{RP}^2)$ by~\cite[Proposition 23]{GG4}.
\end{rem}

\section{Proof of~\reth{result2}}\label{sec:res2}

The main goal of this section is to prove \reth{result2}. Its proof will follow from Propositions \ref{prop1}-\ref{prop3}. In all of the enunciates $X$ denotes a closed surface different of $\mathbb{S}^2$ that admits a free involution $\tau$, the orbit space is denoted by $X_\tau$ and $\theta_\tau\colon \pi_1(X_\tau) \to \mathbb{Z}_2$ is the homomorphism defined in~(\ref{seq:theta}). Recall that $X_\tau$ is also a closed surface. In what follows, we will distinguish the following three possibilities, and in each case, we define an element $\delta \in \pi_1(X_\tau)$:

\begin{enumerate}[(I)]

\item\label{it:I} $X_\tau$ is a closed orientable surface of genus $m\geq 1$, $$\pi_1(X_\tau) = \left\langle a_1, a_2 , \ldots , a_{2m-1} , a_{2m} \, : \, [a_1,a_2] \ldots [a_{2m-1},a_{2m}] = 1 \right\rangle \text{ and } \delta = 1.$$

\item\label{it:II} $X_\tau$ is a closed non-orientable surface of genus $2m+2$, $m\geq 0$, $$\pi_1(X_\tau) = \left\langle u , v , a_1, a_2 , \ldots , a_{2m-1} , a_{2m} \, : \, uvuv^{-1} [a_1,a_2] \ldots  [a_{2m-1},a_{2m}] = 1 \right\rangle \text{ and } \delta = u.$$

\item\label{it:III} $X_\tau$ is a closed non-orientable surface of genus $2m+1$, $m\geq 1$, $$\pi_1(X_\tau) = \left\langle c , a_1, a_2 , \ldots , a_{2m-1} , a_{2m} \, : \, c^2 [a_1,a_2] \ldots [a_{2m-1},a_{2m}] = 1 \right\rangle \text{ and } \delta = c.$$

\end{enumerate}

\begin{rem}\label{rm:delta}
Recall that $\hat{\theta_\tau}\colon\thinspace H_1(X_\tau) \to \mathbb{Z}_2$ is the factorization of the homomorphism $\theta_\tau$ to the abelianizition of $\pi_1(X_\tau)$ and $\hat{\delta}$ is the only-non trivial element of order two of $H_1(X_\tau)$ if $X_\tau$ is non-orientable. So, with the notation of (\ref{it:II}) and (\ref{it:III}) above, we have that $\hat{\theta_\tau}(\hat{\delta}) = \theta_\tau(\delta)$.
\end{rem}

Let $n$ be a positive integer. With the notation of (\ref{it:I})-(\ref{it:III}), in order to estabilish \reth{result2}, by Remark~\ref{rm:delta}, it suffices to prove the following propositions:

\begin{prop}\label{prop1}
If $\theta_\tau(\delta) = \overline{0}$, then the triple $(X,\tau;\mathbb{R}^2)$ does not have the $n$-split BUP and the $n$-non-split BUP.
\end{prop}

\begin{prop}\label{prop2}
If $\theta_\tau(\delta) = \overline{1}$ and $n$ is even, then the triple $(X,\tau;\mathbb{R}^2)$ has the $n$-split BUP, but it does not have the $n$-non-split BUP.
\end{prop}

\begin{prop}\label{prop3}
If If $\theta_\tau(\delta) = \overline{1}$ and $n$ is odd, then the triple $(X,\tau;\mathbb{R}^2)$ has the $n$-BUP.
\end{prop}

The remaining of this paper is devoted to prove Propositions \ref{prop1}-\ref{prop3}. We use without comments the notations estabilished in Sections~\ref{sec:3}.

\begin{proof}[Proof of Proposition \ref{prop1}]
The case $n=1$ is a direct consequence of \cite[Theorem~2.5]{G}. So suppose that $n \geq 2$. For each generator $g \in \pi_1(X_\tau)$, let $i(g) \in \{ 0,1 \}$ such that $\theta_\tau(g) = \overline{i(g)}$. Consider the $2n$-braid $\omega_{n}=\prod^{n-1}_{j=0}\prod^{2n-1-j}_{i=n-j}\sigma_{i}$ ($\omega_4$ is geometrically represented in Figure~\ref{exe}(b)). Note that $w_n \in B^2_{n,n} \setminus B_{n,n}$ and therefore $\pi(w_n)=\overline{1}$. We define $\Psi\colon\thinspace \pi_1(X_\tau) \to B^2_{n,n}$ by $\Psi(g) = \omega_n^{i(g)}$. Using the hypothesis $\theta_\tau(\delta) = \overline{0}$, in each case (\ref{it:I})-(\ref{it:III}) listed in the beginning of the section we may check that $\Psi$ is a well-defined homomorphism such that the diagram~(\ref{diagram}) is commutative. If $a \in \ker{\theta_\tau}$, then $\Psi(a)$ is a even power of $\omega_n$ which is an element of $P_{2n}$. Thus $\Psi(\ker{\theta_\tau}) \subset P_{2n}$. Follows by \reth{diagram} item~(\ref{it:c}) that $(X,\tau;\mathbb{R}^2)$ does not have the $n$-split BUP. Now, let us show that $(X,\tau;\mathbb{R}^2)$ does not have the $n$-non-split BUP. Once more, we may check that $\Psi'\colon\thinspace \pi_1(X_\tau) \to B^2_{n,n}$ defined by $\Psi'(g) = \left( \sigma_1 \omega_n \right)^{i(g)}$  extends to a homomorphism such that the diagram~(\ref{diagram}) is commutative. Since $\theta_\tau$ is an epimorphism, there exists an generator $g_1$ of $\pi_1(X_\tau)$ such that $\theta_\tau(g_1)=\overline{1}$. Then $g_1^2 \in \ker{\theta_\tau}$ and $\Psi'(g_1^2) = (\sigma_1 \omega_n)^2 \notin P_{2n}$. Thus $\Psi'(\ker{\theta_\tau}) \not\subset P_{2n}$ which ends the proof by~\reth{diagram} item~(\ref{it:c}).
\end{proof}

\begin{proof}[Proof of Proposition \ref{prop2}]
The hypothesis $\theta_\tau(\delta) = \overline{1}$ implies that $(X,\tau;\mathbb{R}^2)$ has the $1$-BUP by~\cite[Theorem~2.5]{G}. From this and using analogous arguments that we used in the proof of Proposition~\ref{propB} we conclude that $(X,\tau;\mathbb{R}^2)$ has the $n$-split BUP for all $n$ (in particular for $n$ even). Now, let us prove that $(X,\tau;\mathbb{R}^2)$ does not have $n$-non-split BUP. We claim that for $n=2$ there exists a homomorphism $\Psi_1\colon\thinspace \pi_1(X_\tau) \to B_{2,2}^2$ such that the diagram~(\ref{diagram}) is commutative. Assume for the moment that the claim holds. Let $\tilde{\Psi} \colon\thinspace B^2_{2,2} \to B^2_{2k,2k}$ be the homomorphism obtaneid by \reth{aumentak}, where $k = n/2$. Then the homomorphism $\Psi\colon\thinspace \pi_1(X_\tau) \to B^2_{2k,2k}$ given by the composition $\Psi = \tilde{\Psi} \circ \Psi_1$ makes the diagram~(\ref{diagram}) commutative. By \reth{diagram} item~(\ref{it:c}), the triple $(X,\tau;\mathbb{R}^2)$ does not have $n$-split BUP (if $\Psi \left( \ker{\theta_\tau}  \right) \subset P_{2n}$) or the triple does not have $n$-non-split BUP (if $\Psi \left( \ker{\theta_\tau} \right) \not\subset P_{2n}$). However, we already showed in the first part of the proof that $(X,\tau;\mathbb{R}^2)$ has the $n$-split BUP and therefore we conclude that the triple does not have the $n$-non-split BUP.

In order to complete the proof, we now prove the claim, \emph{i.e.}, let us define a homomorfism $\Psi_1\colon\thinspace \pi_1(X_\tau) \to B_{2,2}^2$ such that the diagram~(\ref{diagram}) is commutative. We will divide the proof into the mutually disjoint cases (\ref{it:i})-(\ref{it:iii}) as below. We will make use the following elements of $B^2_{2,2} \setminus B_{2,2}$: $\Delta_4$ (the half twist as defined in (\ref{def:garside})) and $\Omega = \sigma_2 \sigma_3 \sigma_1^{-1} \sigma_2^{-1}$. Using~(\ref{Delta}), we have
\begin{equation}\label{hom.bemdef}
	\Delta_{4}\Omega\Delta^{-1}_{4}=\Delta_{4}(\sigma_{2}\sigma_{3}\sigma^{-1}_{1}\sigma^{-1}_{2})\Delta^{-1}_{4}=\sigma_{2}\sigma_{1}\sigma^{-1}_{3}\sigma^{-1}_{2}=\Omega^{-1}.
\end{equation}
For each $g \in \pi_1(X_\tau)$, let $i(g) \in \{ 0,1 \}$ such that $\theta_\tau(g)=\overline{i(g)}$.

\begin{enumerate}[(i)]

\item\label{it:i} $\pi_1(X_\tau)$ is as in~(\ref{it:II}). We define 
$$\Psi_1 \colon\thinspace \begin{cases} 
u \mapsto \Omega; \\
v \mapsto \Delta_4 \Omega^{1-i(v)}; \\
a_k \mapsto \Omega^{i(k)} \text{ for all } k.
\end{cases}$$
By~(\ref{hom.bemdef}), we have
\begin{align*}
\Psi_1 \left( uvuv^{-1} [a_1,a_2] \ldots  [a_{2m-1},a_{2m}] \right) 
& = \Omega \Delta_4 \Omega^{1-i(v)} \Omega ( \Delta_4 \Omega^{1-i(v)} )^{-1} [ \Omega^{i(a_1)} , \Omega^{i(a_2)} ] \ldots [ \Omega^{i(a_{2m-1})} , \Omega^{i(a_{2m})} ] \\ 
& = \Omega \Delta_{4}\Omega\Delta^{-1}_{4}  = 1.
\end{align*}
So $\Psi_1$ extends to a homomorphism from $\pi_1(X_\tau)$ to $B^2_{2,2}$. By the hypothesis, we have that $\theta_\tau(u) = \overline{1}$ and from this we may conclude that the diagram~(\ref{diagram}) is commutative.

\item\label{it:ii} $\pi_1(X_\tau)$ is as in~(\ref{it:III}) and $\theta_\tau(a_1) = \overline{1}$ or $\theta_\tau(a_2) = \overline{1}$. We define 
$$\Psi_1 \colon\thinspace \begin{cases} 
	c \mapsto \Omega; \\
	a_1 \mapsto \Delta_4^{i(a_2)} \Omega^{i(a_1)+i(a_2)-2}; \\
	a_2 \mapsto \Delta_4^{1 - i(a_2)} \Omega;\\
	a_k \mapsto \Omega^{i(k)} \text{ for } k \neq 1,2.
\end{cases}$$
Suppose first that $\theta_\tau(a_1) = \overline{1}$ and $\theta_\tau(a_2)=\overline{0}$. By~(\ref{hom.bemdef}) we have
\begin{align*}
\Psi_1(c^2 [a_1,a_2] [a_3,a_4] \ldots [a_{2m-1},a_{2m}]) & = \Omega^2 [\Omega^{-1}, \Delta_4 \Omega] [ \Omega^{i(a_3)} , \Omega^{i(a_4)} ] \ldots [ \Omega^{i(a_{2m-1})} , \Omega^{i(a_{2m})} ] \\
& = \Omega \Delta_{4}\Omega\Delta^{-1}_{4}  = 1. 
\end{align*}
Once more $\Psi_1$ extends to a homomorphism from $\pi_1(X_\tau)$ to $B^2_{2,2}$ and the diagram~(\ref{diagram}) is commutative because $\theta_\tau(c) = \overline{1}$ by hypothesis. The proof of the cases ($\theta_\tau(a_1) = \overline{0}$ and $\theta_\tau(a_2)=\overline{1}$) and ($\theta_\tau(a_1) = \overline{1}$ and $\theta_\tau(a_2)=\overline{1}$) are completely similar and we omit the details.

\item\label{it:iii} $\pi_1(X_\tau)$ is as in~(\ref{it:III}) and $\theta_\tau(a_1) = \theta_\tau(a_2) = \overline{0}$. We leave to the reader to verify that the braids $\sigma_2 \Omega^{-1} \sigma_2^{-1}$ and $\sigma_2 \sigma_3^2 \sigma_2^{-1}$ belongs to $B_{2,2}$. We define 
$$\Psi_1 \colon\thinspace \begin{cases} 
	c \mapsto \Omega; \\
	a_1 \mapsto \sigma_2 \Omega^{-1} \sigma_2^{-1}; \\
	a_2 \mapsto \sigma_2 \sigma_3^2 \sigma_2^{-1};\\
	a_k \mapsto \Omega^{i(k)} \text{ for } k \neq 1,2.
\end{cases}$$
 From the Artin relations in $B_4$ we deduce that $\sigma_{i}^{\varepsilon} \sigma_{i+1} \sigma_{i}^{-\varepsilon} = \sigma_{i+1}^{-\varepsilon} \sigma_{i} \sigma_{i+1}^{\varepsilon}$, for $i=1,2$ and $\varepsilon = -1,1$. Using this and $\sigma_1 \sigma_3 = \sigma_3 \sigma_1$ we have
\begin{align*}
\Psi_1(c^2 [a_1,a_2] [a_3,a_4] \ldots [a_{2m-1},a_{2m}])
& = \Omega^2 [ \sigma_2 \Omega^{-1} \sigma_2^{-1} , \sigma_2 \sigma_3^2 \sigma_2^{-1} ] [ \Omega^{i(a_3)} , \Omega^{i(a_4)} ] \ldots [ \Omega^{i(a_{2m-1})} , \Omega^{i(a_{2m})} ] \\
& = \Omega^2 \sigma^{2}_{2}\sigma_{1}\sigma^{-1}_{3}(\sigma^{-1}_{2}\sigma^{2}_{3}\sigma_{2})\sigma_{3}\sigma^{-1}_{1}\sigma^{-1}_{2}\sigma^{-2}_{3}\sigma^{-1}_{2} \\ 
&= \Omega^2 \sigma^{2}_{2}\sigma_{1}\sigma^{-1}_{3}(\sigma_{3}\sigma^{2}_{2}\sigma^{-1}_{3})\sigma_{3}\sigma^{-1}_{1}\sigma^{-1}_{2}\sigma^{-2}_{3}\sigma^{-1}_{2}\\
& = \Omega^2 \sigma^{2}_{2}(\sigma_{1}\sigma^{2}_{2}\sigma^{-1}_{1})\sigma^{-1}_{2}\sigma^{-2}_{3}\sigma^{-1}_{2} \\
& = \Omega^2 \sigma^{2}_{2}(\sigma^{-1}_{2}\sigma^{2}_{1}\sigma_{2})\sigma^{-1}_{2}\sigma^{-2}_{3}\sigma^{-1}_{2} \\
& = \Omega^2 \sigma_{2}\sigma^{2}_{1}\sigma^{-2}_{3}\sigma^{-1}_{2} = \Omega^2 \Omega^{-2} = 1.
\end{align*}
As in the previous cases we have that the diagram~(\ref{diagram}) is commutative.
\end{enumerate} 
\end{proof}

\begin{proof}[Proof of Proposition \ref{prop3}]
If $n=1$, then result follows directly from~\cite[Theorem~2.5]{G}. So let $n$ be an odd positive integer different from $1$. In order to simplify the proof, we write
$$ \pi_1(X_\tau) =   \left\langle \alpha , \beta , a_1, a_2 , \ldots , a_{2m-1} , a_{2m} \, : \, \alpha \beta \alpha \beta^{-1} [a_1,a_2] \ldots  [a_{2m-1},a_{2m}] = 1 \right\rangle, $$
where $\begin{cases} \alpha = u \\ \beta = v \end{cases}$ if $\pi_1(X_\tau)$ is as in~(\ref{it:II}) and $\begin{cases} \alpha = c \\ \beta = 1 \end{cases}$ if $\pi_1(X_\tau)$ is as in~(\ref{it:III}). For each generator $g \in \pi_1(X_\tau)$, let $i(g) \in \{ 0,1 \}$ such that $\theta_\tau(g) = \overline{i(g)}$. By hypothesis, we have $i(\alpha) = \overline{1}$.  We argue by contradiciton. Suppose that the triple $(X,\tau,\mathbb R^{2})$ does not have $n$-BUP. By \reth{diagram} there exists a homomorphism $\Psi\colon\thinspace \pi_1(X_\tau) \to B_{n,n}^2$ such that the diagram~(\ref{diagram}) is commutative. From this and Remark~\ref{rem:delta}, there exists braids $b_\alpha, b_\beta, b_1, b_2, \ldots, b_{2m-1},b_{2m} \in B_{n,n}$ such that
$$
\Psi\colon\thinspace \begin{cases}
\alpha \mapsto b_\alpha \Delta \\
\beta \mapsto b_\beta \Delta^{i(\beta)} \\
a_k \mapsto b_k  \Delta^{i(a_k)},
\end{cases}
$$
where $\Delta = \Delta_{2n}$ is the half twist. Applying $\Psi$ in the relation of the presentation of $\pi_1(X_\tau)$  and using the fact that $B_{n,n}$ is a normal subgroup of $B^2_{n,n}$, we obtain the following equation in $B_{n,n}$:
\begin{multline}\label{eq:propIII}
b_\alpha \Bigl( \Delta b_\beta \Delta^{-1} \Bigr) \Delta^2 \Bigl( \Delta^{i(\beta)-1} b_\alpha \Delta^{-i(\beta)+1} \Bigr) b_\beta^{-1} . 
 \Biggl( b_1 \Bigl( \Delta^{i(a_1)} b_2 \Delta^{-i(a_1)} \Bigr)  \Bigl( \Delta^{i(a_2)} b_1 \Delta^{-i(a_2)} \Bigr)^{-1} b_2^{-1} \Biggr) \ldots \\
. \Biggl( b_{2m-1} \Bigl( \Delta^{i(a_{2m-1})} b_{2m} \Delta^{-i(a_{2m-1})} \Bigr)  \Bigl( \Delta^{i(a_{2m})} b_{2m-1} \Delta^{-i(a_{2m})} \Bigr)^{-1} b_{2m}^{-1} \Biggr)= 1,
\end{multline}
where $\Delta^2 = \Delta^2_{2n}$ is the full twist. Follows by Proposition~\ref{prop:B_{n,n}} that it is well defined the homomorphism $\varepsilon\colon\thinspace B_{n,n} \to \mathbb{Z}_2$ given in the generators by
$$\varepsilon \colon\thinspace \begin{cases}
\sigma_k \mapsto \overline{0}, \text{ if } 1 \leq k \leq n-1 \text{ or } n+1 \leq k \leq 2n-1, \\
A_{i,j} \mapsto \overline{0}, \text{ if } 1\leq i< j\leq n \text { or } n+1\leq i < j \leq 2n  , \\
A_{i,j} \mapsto \overline{1}, \text{ if } 1\leq i \leq n \text{ and } n +1 \leq j\leq 2n .
\end{cases}$$
By (\ref{Delta}) we have that
$$\varepsilon(\Delta \sigma_k \Delta^{-1}) = \varepsilon (\sigma_{2n-k}) = \overline{0} = \varepsilon(\sigma_k)$$
for all $k$. Using the equality above and Remark~\ref{rem:Aij}, we conclude
$$\varepsilon(\Delta A_{i,j} \Delta^{-1}) = \varepsilon (A_{i,j})$$
for all $i$ and $j$.
Since the target of $\varepsilon$ is $\mathbb{Z}_2$, we conclude that
\begin{center}
$\varepsilon(b) + \varepsilon( \Delta b \Delta^{-1}) = \overline{0}$ for all $b \in B_{n,n}$.
\end{center}
Applying $\varepsilon$ in both sides of equation~(\ref{eq:propIII}), using the property above and~(\ref{full}), we obtain:
\begin{equation*}
\overline{0} = \varepsilon(\Delta^2) = \displaystyle \sum_{1 \leq i < j \leq 2n} \varepsilon ( A_{i,j} ) = \sum_{j=n+1}^{2n}  \sum_{i=1}^{n} \overline{1} = \overline{n^2}.
\end{equation*}
The last equality is an absurd because $n$ is odd. Thus $(X,\tau;\mathbb{R}^2)$ has the $n$-BUP.
\end{proof}

\begin{section}{Some additional results about the triples $(X,\tau;Y)$}\label{sc:examples}

In this section, we obtain some partial results about the classification of the triples $(X,\tau;Y)$ in relation to the $n$-BUP, where $X$ and $Y$ are closed surfaces, $X \neq \mathbb{S}^2$. We stated these results in the following two propositons. The first one show us that there exists triples that do not have the $n$-BUP and in the second one we prove that the opposite situation is also true. We start with the following lemma:

\begin{lem}\label{lem:I} Let $X$ and $Y$ be closed surfaces. If a triple $(X,\tau;\mathbb{R}^2)$ does not have the $n$-split-BUP (resp.\ $n$-non-split BUP), then $(X,\tau;Y)$ also does not have the $n$-split-BUP (resp.\ $n$-non-split BUP).
\end{lem}

\begin{proof}
Let $(X,\tau;\mathbb{R}^2)$ be a triple that does not have the $n$-split BUP (resp.\ $n$-non-split). As used in the proof of Theorem~\ref{th:diagram}, there exists a continuous map $\Phi_1 \colon\thinspace X \to F_n(Y)$ (resp.\ $\Phi_2\colon\thinspace X \to D_n(Y)$) such that $\Phi_i(\tau(x)) \neq \Phi_i(x)$ for all $x \in X$ and $i\in\{1,2\}$. Since $Y$ is a closed surface, there exists an embedding $\iota\colon\thinspace \mathbb{R}^2 \to Y$ which induces  the embeddings $\iota_1\colon\thinspace F_n(\mathbb{R}^2) \to F_n(Y)$ and $\iota_2\colon\thinspace D_n(\mathbb{R}^2) \to D_n(Y)$. Then the maps $\Phi^\prime_i = \iota_i \circ \Phi_i$ satisfies $\Phi^\prime_i(\tau(x)) \neq \Phi^\prime_i(x)$ for all $x \in X$ and $i \in \{1,2\}$. Therefore the triple $(X,\tau;Y)$ does not have the $n$-split BUP (resp.\ $n$-non-split BUP).
\end{proof}

Recall that $X_\tau$ denotes the orbit space of $X$ by the involution $\tau$ and $\hat{\delta}$ is defined after the sequence~(\ref{seq:theta}). The following proposition follows imediately from Theorem~\ref{th:result2} and Lemma~\ref{lem:I}

\begin{prop}\label{prop_examples_1} Let $n \geq 1$ and $X \neq \mathbb{S}^2$ and $Y$ be closed surfaces, where $X$ is equipped with a free involution $\tau$. The following conditions are true:
\begin{enumerate}
\item If $X_\tau$ is a closed orientable surface or $X_\tau$ is a closed non-orientable surface and $\hat{\theta_\tau}(\hat{\delta})=\overline{0}$, then the triple $(X,\tau;Y)$ does not have the $n$-split BUP and it does not have the $n$-non-split BUP.
	
\item If $X_\tau$ is a closed non-orientable surface, $\hat{\theta_\tau}(\hat{\delta})=\overline{1}$ and $n$ is even, then the triple $(X,\tau;Y)$ does not have the $n$-non-split BUP.
\end{enumerate}
\end{prop}

Lastly, using the same idea of the proof of Proposition~\ref{propB}, we may conclude that if a triple $(X,\tau;Y)$ has the $1$-BUP, then the same triple has the $n$-split BUP for all positive integer $n$. Follows directly from this fact and \cite[Theorem~12 and Proposition~32]{GG0} the following result. Recall that the homomorphism $\theta_\tau$ was defined after the sequence~(\ref{seq:theta}).

\begin{prop}\label{prop_examples_2}
Let $X$ and $Y$ be closed surfaces different of $\mathbb{S}^2$, where $X$ is eqquiped with a free involution $\tau$ and $Y$ is orientable. In relation to the presentation of $\pi_1(X_\tau)$ given in Section~\ref{sec:res2}, items~(\ref{it:II}) and~(\ref{it:III}), suppose that one of the following conditions is true: 
\begin{enumerate}
\item $X_\tau$ is the Klein bottle and $\theta_\tau(u)=\overline{1}$.

\item $X_\tau$ is the closed non-orientable surface of genus $3$ and $\theta_\tau(c)= \theta_\tau(a_i) = \overline{1}$ for some $i \in \{ 1,2 \}$.
\end{enumerate}	
Then the triple $(X,\tau;Y)$ has the $n$-split BUP for all $n \geq 1$.
\end{prop}
	
\end{section}

We hope to extend Propositions~\ref{prop_examples_1} and~\ref{prop_examples_2} and obtain a full classification of the triples $(X,\tau;Y)$ in relation to the $n$-BUP, and we intend to present these results soon. In order to accomplish this, we are using the techniques developed in this article, in particular the intermediary braid groups $B_{n,n}^2(Y)$ and Theorem~\ref{th:diagram}.

\begin{section}{Acknowledgements}
The work on this paper was developed during the Postdoctoral Internship of the first author at IME-USP from March 2020 to August 2021 that was supported  by the Capes/INCTMat project nº 88887.136371/2017-00-465591/2014-0. The authors would like to thanks the professors Daciberg Lima Gonçalves and Daniel Vendrúscolo for their great contributions.
\end{section}

\end{document}